\documentclass{article}
\usepackage[noadjust]{cite}

\usepackage{xcolor}
\usepackage{amsthm}
\usepackage{amsfonts}
\usepackage{amssymb}
\usepackage{amsgen}
\usepackage{amsmath}
\usepackage{amsopn}
\usepackage{verbatim}
\usepackage{xypic}
\usepackage{pgf}
\usepackage{xspace}
\usepackage{multicol}
\usepackage{makeidx}
\usepackage{eepic}
\usepackage{upref}
\usepackage{pgf}
\usepackage{tikz}
\usepackage[normalem]{ulem}
\usepackage{shuffle,yfonts}
\DeclareFontFamily{U}{shuffle}{}
\DeclareFontShape{U}{shuffle}{m}{n}{ <-8>shuffle7 <8->shuffle10}{}

\allowdisplaybreaks

\newcommand{\FES}{\mathsf {FES}}

\newcommand{\FMTV}{\mathsf {FMT}}

\newcommand{\SMSV}{\mathsf {SMS}}

\newcommand{\ES}{\mathsf {ES}}

\newcommand{\MTV}{\mathsf {MT}}

\newcommand{\SMTV}{\mathsf {SMT}}
\newcommand{\sha}{\shuffle}

\newcommand{\Sy}{{\mathcal S}}

\newcommand{\ga}{\alpha}

\newcommand\om{{\omega}}
\newcommand\tom{{\widetilde{\omega}}}

\newcommand\eps{{\varepsilon}}
\newcommand{\bfk}{{\bf k}}
\newcommand{\bfl}{{\bf l}}

\newcommand{\bfs}{{\boldsymbol{\sl{s}}}}

\newcommand\bfeps{{\boldsymbol \varepsilon}}

\newcommand{\calA}{\mathcal{A}}

\newcommand{\calP}{\mathcal{P}}

\allowdisplaybreaks

\newcommand{\N}{\mathbb{N}}
\newcommand{\Z}{\mathbb{Z}}
\newcommand{\Q}{\mathbb{Q}}

\newcommand{\ol}{\overline}

\def\ppmod#1{{\ (\rm{mod}\ 2)}}

\theoremstyle{plain}
\newtheorem{thm}{Theorem}[section]
\newtheorem{lem}[thm]{Lemma}
\newtheorem{conj}[thm]{Conjecture}

\newtheorem{cor}[thm]{Corollary}
\newtheorem{prop}[thm]{Proposition}

\theoremstyle{definition}

\newtheorem{rem}[thm]{Remark}

\begin{document}
\title{Finite and Symmetric Multiple $T$-Values}

\author{Aaron Cheng\footnote{Email: aaron.cheng.28@bishops.com.} and Jianqiang Zhao\footnote{Email: zhaoj@ihes.fr, corresponding author.}\\
\ \\
Department of Mathematics, The Bishop's School, La Jolla, CA 92037, USA}

\date{}

\maketitle

\medskip
\noindent
\textbf{Abstract.}
The multiple $T$-values (MTVs), first studied by Kaneko and Tsumura, are a variation of the multiple zeta values (MZVs) with restricted product structure. Motivated by a deep conjecture of Kaneko and Zagier relating finite MZVs and symmetric MZVs, which was extended to Euler sums by Zhao, we study finite and symmetric multiple $T$-values. In particular, we show that finite MTVs satisfy Hoffman-type duality relations at low height, confirming several conjectures of the second author and discovering new families of identities. In proving relations among symmetric MTVs, our work builds on Xu and Zhao's theory of multiple mixed values, the relations of double zeta values discovered by Gangl, Kaneko and Zagier, and the (weighted) sum formulas of double Euler sums discovered by Berger et al. We then use generating functions derived from the integral structure of MTVs to establish the corresponding relations for finite MTVs. These results aid in the computation of dimensions of the $\Q$-vector space spanned by finite and symmetric MTVs (modulo $\zeta(2)$ products), providing strong evidence for an isomorphism between the two spaces.



\section{Introduction and some preliminary results}
\subsection{Multiple zeta value and Euler sums}
Euler's study of double zeta values is well-documented in his correspondence with Goldbach in the mid-18th century. After over 200 years of dormancy, their study was revived by Hoffman \cite{Hoffman1992} and Zagier \cite{Zagier1994} independently in the early 1990s. For any $d$-tuples of positive integers $\bfs=(s_1,\dots,s_d)\in\N^d$ with $s_1>1$, one defines the \emph{multiple zeta value} (MZV)
\begin{equation}\label{defn:MZV}
\zeta(\bfs):=\sum_{n_1>\dots>n_d>0}  \frac{1}{n_1^{s_1}\cdots n_d^{s_d}}.
\end{equation}
We call such $\bfs$ an admissible index. If $s_1=1$ then it is non-admissible because the series diverges in this case. For convenience, we call $d$ the depth and $|\bfs|:=s_1+\dots+s_d$ the weight. It is easy to see from \eqref{defn:MZV} that the product of two MZVs can be expressed as a sum of MZVs of the same weight but possibly different depths, which we call the stuffle relations. For example,
\begin{equation*}
    \zeta(2)\zeta(3)=\zeta(2,3)+\zeta(3,2)+\zeta(5).
\end{equation*}

One of the most important properties of MZVs is that they possess iterated integral expressions. Kontsevich first noticed that
for all admissible $\bfs$,
\begin{equation}\label{equ:MZVitIntegral}
\zeta(\bfs)=\int_0^1  \left(\frac{dt}{t}\right)^{s_1-1}\frac{dt}{1-t} \cdots \left(\frac{dt}{t}\right)^{s_d-1}\frac{dt}{1-t}.
\end{equation}
See \cite[\S 2.1]{Zhao2016} for more details. Using Chen's theory of iterated integrals developed in \cite{KTChen1971,KTChen1977}, one sees that the product of two MZVs can be expressed as a sum of MZVs of the same weight and \emph{same depth} according to the shuffle product relations of iterated integrals.

Although Euler mistakenly treated divergent double zeta values corresponding to non-admissible indices
as if they converged, his method and ideas were generalized to multiple zeta values and made rigorous in the influential work \cite{IKZ2006} of Ihara, Kaneko, and Zagier. In particular, they discovered two ways to regularize divergent MZVs: one by the definition of the series in \eqref{defn:MZV} using the stuffle product, and the other by the integral expression in \eqref{equ:MZVitIntegral} using the shuffle product.

All of the above have been generalized to higher levels by Racinet \cite{Racinet2002} by inserting powers of the $N$-th roots of unity in the numerators of \eqref{defn:MZV}.
In particular, when $N=2$ these are sometimes called Euler sums or alternating MZVs. For $\bfs=(s_1,\dots,s_d)\in\N^d$ and
$\bfeps=(\eps_{1}, \dots,\eps_d)\in\{\pm 1\}^d$, one defines
\begin{equation}\label{defn:ES}
\zeta(\bfs;\bfeps):=\sum_{n_1>\dots>n_d>0}  \frac{\eps_1^{n_1}\cdots \eps_d^{n_d}}{n_1^{s_1}\cdots n_d^{s_d}}.
\end{equation}
For simplicity, one often omits $\bfeps$ with the standard short-hand notation of writing $\ol{s_j}$ if $\eps_j=-1$. For example,
the second author proved the following identity in \cite{Zhao2010a}:
\begin{equation}\label{equ:2bar1}
  8^\ell   \zeta(\{\bar2,1\}^\ell)= 8^\ell  \zeta(\{2,1\}^\ell;\{-1,1\}^\ell)=\zeta(\{3\}^\ell)
\end{equation}
for all $\ell\in\N$, where $\{\bfs\}^\ell$ means that the string $\bfs$ repeats exactly $\ell$ times. Moreover, if $s_1=\eps_1=1$
then, similar to MZVs, there are two (generally different) regularized values: $\zeta_\ast(\bfs;\bfeps)$ by the stuffle product and $\zeta_\sha(\bfs;\bfeps)$ by the shuffle product (see \cite[\S 13.3]{Zhao2016}).

\subsection{A key conjecture of Euler sums}
We can further define the so-called finite Euler sums $\zeta_\calA(\bfs;\bfeps)$ as follows. For any $\bfs\in\N^d$ and $\bfeps\in\{\pm 1\}^d$ we define
\begin{equation*}
\zeta_\calA(\bfs;\bfeps):=\left( \sum_{p>n_1>\dots>n_d>0}  \frac{\eps_1^{n_1}\cdots \eps_d^{n_d}}{n_1^{s_1}\cdots n_d^{s_d}} \pmod{p}\right)_{p: \calP}\in \calA:=\frac{\prod_{p\in\calP} \Z/p\Z }{\sum_{p\in\calP} \Z/p\Z}
\end{equation*}
where $\calP$ is the set of primes. These infinite tuples then form a $\Q$-vector space denoted by $\FES$. The second author proposed the following key conjecture
(see \cite[Conjecture~8.6.9]{Zhao2016}) by extending a similar conjecture of Kakeko and Zagier (see \cite[Main Conjecture]{KanekoZa2013}).

\begin{conj}\label{conj:KanekoZagierAltVersion}
For any $w\in\N$, let $\FES_{w}$ (resp.\ $\ES_w$) be the $\Q$-vector space generated by
all finite Euler sums (resp.\ Euler sums) of weight $w$. Then, there is an isomorphism:
\begin{align*}
f_\ES: \FES_{w} & \longrightarrow \frac{\ES_w}{\zeta(2)\ES_{w-2}} , \\
 \zeta_{\calA}(\bfs;\bfeps) & \longmapsto \zeta_\sha^\Sy(\bfs;\bfeps)  \pmod{\zeta(2)}
\end{align*}
where
\begin{align*}
\zeta_\sha^\Sy(\bfs;\bfeps):=&\sum_{i=0}^d
\left(\prod_{j=1}^i (-1)^{s_j} \eps_j \right) \zeta_\sha(s_i,\dots,s_1;\eps_i,\dots,\eps_1) \zeta_\sha(s_{i+1},\dots,s_d;\eps_{i+1},\dots,\eps_d)
\end{align*}
where $\zeta_\sha(\bfs;\bfeps)$ is the shuffle-regularized value (which is a non-constant polynomial if $s_1=\eps_1=1$ and is equal to $\zeta(\bfs;\bfeps)$ otherwise).
\end{conj}
From \cite[Prop. 8.6.8]{Zhao2016} we know that the symmetric Euler sum $\zeta_\sha^\Sy(\bfs;\bfeps)$ is always a constant real number. Moreover,
one can replace $\zeta_\sha^\Sy(\bfs;\bfeps)$ by the stuffle regularized value $\zeta_\ast^\Sy(\bfs;\bfeps)$
since the difference always lies in $\zeta(2)\ES_{w-2}$.

\subsection{Parity variations of multiple zeta values}
In recent years, some variations of MZVs have been defined by restricting the parity patterns of the summation indices in \eqref{defn:MZV}.
In \cite{KanekoTs2020}, Kaneko and Tsumura introduced \emph{multiple $T$-values} (MTVs), defined by
\begin{align}\label{defn:MTV}
T(\bfs):= \sum_{\substack{n_1>\dots>n_d>0\\ n_j\equiv d-j+1\ppmod{2}}}  \frac{2^d}{n_1^{s_1}\cdots n_d^{s_d}},
\end{align}
for all admissible $\bfs=(s_1,\dots,s_d)\in\N^d$. It is nontrivial to see that MTVs can be expressed using iterated integrals
\begin{equation}\label{equ:MTVitIntegral}
T(\bfs)=\int_0^1  \left(\frac{dt}{t}\right)^{s_1-1}\frac{2dt}{1-t^2} \cdots \left(\frac{dt}{t}\right)^{s_d-1}\frac{2dt}{1-t^2}.
\end{equation}
Furthermore, they satisfy both the duality relations (see \cite[Thm. 3.1]{KanekoTs2020}) and the integral shuffle relations (see \cite[Thm. 2.1]{KanekoTs2020}), which is the main motivation to define MTVs in the first place.
Motivated by this, Xu and the second author \cite{XuZhao2020a} defined \emph{multiple $S$-values} (MSVs) by
\begin{align*}
S(\bfs):= \sum_{\substack{n_1>\dots>n_d>0\\ n_j\equiv d-j\ppmod{2}}} \frac{2^d}{n_1^{s_1}\cdots n_d^{s_d}},
\end{align*}
i.e., the parity pattern of the summation indices is exactly opposite to that of MTVs. Furthermore, we have developed  a general theory for arbitrary parity patterns in \cite{XuZhao2020a,XuZhao2020c} by defining the multiple mixed values
\begin{equation}\label{defn:MMV}
M(\bfs;\bfeps):=\sum_{m_1>\cdots>m_d>0} \prod_{j=1}^d \frac{(1+\eps_j(-1)^{m_j})}{m_j^{s_j}}
            =\sum_{\substack{n_1>\cdots>n_d\\ 2| n_j \text{ if } \eps_j=1 \\ 2\nmid n_j \text{ if } \eps_j=-1}}  \frac{2^d}{n_1^{s_1}\cdots n_d^{s_d}}
\end{equation}
for all admissible $\bfs\in\N^d$ and $\bfeps\in\{\pm 1\}^d$.
According to \cite[(2.1)]{XuZhao2020a}\footnote{Notice that the order of summation indices in the definition of multiple mixed value in \cite{XuZhao2020a} is increasing, which is opposite to the current paper.}, if $(s_1,\eps_1)\ne (1,1)$ then
\begin{equation*}
M(s_1,\dots,s_d;\eps_1,\eps_2,\dots,\eps_d)= \int_0^1 \om_0^{s_1-1}\om_{\eps_1\eps_2}\cdots\om_0^{s_{d-1}-1}\om_{\eps_{d-1}\eps_d}\om_0^{s_d-1}\om_{\eps_d}
\end{equation*}
where
\begin{align*}
 \om_0=\frac{dt}{t}, \quad \om_1=\frac{2t\, dt}{1-t^2}, \quad \om_{-1}=\frac{2dt}{1-t^2}.
\end{align*}
The following relations for multiple mixed values generalize the duality relations for MTVs \cite[Thm. 3.1]{KanekoTs2020}.

\begin{thm} \label{thm:dualMMVo}
Let $\bfk=(k_1,\ldots,k_r)\in\N^r,\bfl=(l_1,\ldots,l_r)\in\N^r$ and $\bfeps\in\{\pm 1\}^r$.
Then
\begin{equation}\label{equ:dualMMV}
 \int_0^1 \om_0^{k_r} \om_{\eps_r}^{l_1} \cdots \om_0^{k_r}\om_{\eps_1}^{l_1}
=\int_0^1  u_{\eps_1}^{l_1} \om_{-1}^{k_1}  \cdots u_{\eps_r}^{l_r} \om_{-1}^{k_r}  ,
\end{equation}
where $u_{-1}=\om_0$ and $u_1=\om_0+\om_1-\om_{-1}$.
\end{thm}
\begin{proof}
This follows immediately from the substitution $t\to \frac{1-t}{1+t}$.
\end{proof}

We note that Theorem \ref{thm:dualMMVo} slightly extends \cite[Thm. 2.13]{XuZhao2020a} by allowing $\eps_1=1$. Note that
$\int_0^1 (\om_1-\om_{-1}) \cdots \om_{-1}^{k_r}$ always converges.

\subsection{Finite and symmetric multiple $T$-values}
In \cite{Zhao2023Dec,Zhao2024a}, guided by the set-up for MZVs and Euler sums, one of us defined and studied various finite and symmetric MSVs and MTVs. The finite versions $S_\calA$ and $T_\calA$ are defined in the exact same way as finite Euler sums: by truncating the defining series at every prime $p$, and then forming infinite tuples indexed by $p$ such that $S_\calA(\bfs),T_\calA(\bfs)\in\calA$ for all index $\bfs$ (non-admissible ones are included).

The symmetric MSVs and MTVs are defined by using two types of regularized MSVs and MTVs $F_\ast(\bfs)$ and $F_\sha(\bfs)$ ($F=S$ or $T$). More precisely, for any $\bfs=(s_1,\dots,s_d)\in\N^d$ (not necessarily admissible) one can define the symmetric MTV
\begin{equation*}
T_\sharp^\Sy(\bfs):=
\left\{
  \begin{array}{ll}
  \displaystyle \sum_{i=0}^d  \Big(\prod_{\ell=1}^i (-1)^{s_\ell} \Big) T_\sharp(s_i,\dots,s_1)T_\sharp(s_{i+1},\dots,s_d),\quad & \quad \hbox{if $d$ is even;} \\
  \displaystyle \sum_{i=0}^d  \Big(\prod_{\ell=1}^i (-1)^{s_\ell} \Big) S_\sharp(s_i,\dots,s_1)T_\sharp(s_{i+1},\dots,s_d), \quad&\quad \hbox{if $d$ is odd,}
  \end{array}
\right.
\end{equation*}
where, $\sharp=\ast$ or $\sha$, and as usual, $\prod_{\ell=1}^0=1$. $S_\sharp^\Sy(\bfs)$ is defined by the exchange of $S$ and $T$ on the right-hand side of the above.

The main goal of this paper is to prove some families of relations that hold for both finite and symmetric MTVs. For example, the following results follow from Theorem~\ref{thm:SMT121} and Corollary~\ref{cor:TA21_k} since $T_\sha^\Sy(n)=(2-2^{2-n})/(2-2^{1-n})T(n)$ for any positive integer $n>1$.
\begin{thm}
For any odd positive integer $k$ and any nonnegative integer $\ell\le k$, we have
\begin{align*}
T_\sha^\Sy(\{1\}^\ell,2,\{1\}^{k-\ell})=\frac{(-1)^{\ell}}{\ell+1} \binom{k+1}{\ell}  T_\sha^\Sy(2,\{1\}^k)= (-1)^{\ell} \binom{k+2}{\ell+1} \frac{(2^{k+2}-1)}{2(2^{k+1}-1)} T_\sha^\Sy(k+2),\\
T_\calA(\{1\}^\ell,2,\{1\}^{k-\ell})=\frac{(-1)^{\ell}}{\ell+1} \binom{k+1}{\ell}  T_\calA(2,\{1\}^k)= (-1)^{\ell} \binom{k+2}{\ell+1} \frac{(2^{k+2}-1)}{2(2^{k+1}-1)}T_\calA(k+2).
\end{align*}
\end{thm}

This was motivated by \cite[Prop. 3]{Zhao2023Dec}, in which only the first equation for finite MTVs in the second line above was proven.
In view of similar results and conjectures we have obtained so far, in which both finite and symmetric MTVs satisfy essentially the same relations, we now formulate a modified version of Conjecture~\ref{conj:KanekoZagierAltVersion}.

\begin{conj}\label{conj:KanekoZagierMTV}
For any $w\in\N$, let $\SMTV_{w}$ (resp.\ $\FMTV_w$) be the $\Q$-vector space generated by
all symmetric MTVs (resp.\ finite MTVs) of weight $w$. Then, there is an isomorphism between $\Q$-vector spaces:
\begin{align*}
f_\MTV: \FMTV_{w} & \longrightarrow \frac{\SMTV_w}{\zeta(2)\ES_{w-2}\cap \SMTV_w},  \\
 T_{\calA}(\bfs) & \longmapsto T_\sha^\Sy(\bfs).
\end{align*}
\end{conj}

In the appendix, we investigate the space generated by symmetric MTVs and the spaces appearing in Conjecture~\ref{conj:KanekoZagierMTV} for weight up to 10. By numerical computation, we are able to verify this conjecture in this range.

\section{Unit finite and symmetric multiple $T$-values}
An index of the form $(1,\dots,1)$ is called a unit index. Through extensive numerical computation, we have found in \cite{Zhao2023Dec} that the following conjecture concerning such indices should be true.

\begin{conj}\label{conj:unitT} \emph{(\cite[Conj. 5]{Zhao2023Dec})}
Let $\ES_{w}^{(1)}$ be the space generated $\zeta(w)$. For all $k\in\N$ with $k\ge 2$, we have
\begin{equation*}
  T_\calA(\{1\}^k)=\frac{2^{k-1}-1}{2^{k-2}} \beta_k, \qquad
  T_\sha^\Sy(\{1\}^k)\equiv \frac{2^{k-1}-1}{2^{k-2}} \zeta(k) \pmod{\zeta(2)\ES_{k-2}^{(1)}},
\end{equation*}
where $\beta_k=\big(B_{p-k}/k \pmod{p}\big)_{p>k,p\in \calP}$ and $B_n$'s are Bernoulli numbers defined by
\begin{equation*}
    \frac{t}{e^t-1}=\sum_{n\ge 0} B_n \frac{t^n}{n!}.
\end{equation*}
\end{conj}
We will prove the above conjecture in this section by dealing with the finite and the symmetric versions separately.

\begin{rem} (a). Note that $\beta_k$ is the finite analog of $\zeta(k)$.
(b). When $k$ is even the first equation Conjecture~\ref{conj:unitT} holds by \cite[Prop. 7]{Zhao2023Dec}.
(c). Because of the extra factor of $2^d$ in the definition \eqref{defn:MTV}, Conjecture \ref{conj:unitT} has a slightly different form
from \cite[Conj. 5]{Zhao2023Dec}. Also, $\mod \zeta(2)$ was dropped there mistakenly.
\end{rem}

We now show that the second relation in Conjecture~\ref{conj:unitT} holds. The first relation will be proved in Theorem~\ref{thm:unitTA}.

\begin{thm}\label{thm:unitSMT}
For all $k\in\N$, we have
\begin{equation*}
  T_\sha^\Sy(\{1\}^k)=
\left\{
  \begin{array}{ll}
    0, & \hbox{if $k$ is even,}\\
    2\log 2, \phantom{ \displaystyle \frac12} & \hbox{if $k=1$;} \\
    2(1-2^{1-k}) \zeta(k), & \hbox{if $k\ge 2$ and $k$ is odd.}
  \end{array}
\right.
\end{equation*}
In particular,
\begin{equation} \label{equ:T111=dual}
  T_\sha^\Sy(\{1\}^k)\equiv T_\sha^\Sy(k)  \pmod{\zeta(2)\ES_{k-2}^{(1)}}.
\end{equation}
\end{thm}

\begin{proof}
When $k$ is even, by definition
\begin{align*}
    T_\sha^\Sy(\{1\}^k)=&\,  \sum_{j=0}^k  (-1)^j T_\sha(\{1\}^j)T_\sha(\{1\}^{k-j})\\
=&\, \int_0^1 \sum_{j=0}^k  (-1)^j \om_{-1}^j\sha \om_{-1}^{k-j}\\
=&\, \int_0^1 2\om_{-1}^k+  \sum_{j=1}^{k-1}  (-1)^j \om_{-1}^j\sha \om_{-1}^{k-j}  \\
=&\, \int_0^1 2\om_{-1}^k+  \sum_{j=1}^{k-1}  (-1)^j \om_{-1}(\om_{-1}^{j-1}\sha \om_{-1}^{k-j})
+ \sum_{j=1}^{k-1}  (-1)^j \om_{-1}(\om_{-1}^j\sha \om_{-1}^{k-j-1}) =0
\end{align*}
by changing the index $j\to j-1$ in the first sum since $k$ is even. On the other hand, Euler's famous identity for $k=2n$ says
\begin{equation*}
\zeta(2n)=\frac{(-1)^n}{2} \frac{(24\zeta(2))^{n} B_{2n}}{(2n)!}\equiv 0   \pmod{\zeta(2)\ES_{k-2}^{(1)}}
\end{equation*}
since $T(k-2)=(2-2^{3-k}) \zeta(k-2)$. This proves the theorem when $k$ is even.

When $k=1$, by the regularization given in \cite[Prop. 2.8]{XuZhao2020a}
\begin{equation*}
T_\sha^\Sy(1)= T_\sha(1)-S_\sha(1)=2\log 2.
\end{equation*}
When $k>1$ and $k$ is odd, by definition
\begin{align*}
 T_\sha^\Sy(\{1\}^k)=&\, T_\sha(\{1\}^k)+ \sum_{j=1}^k  (-1)^j S_\sha(\{1\}^j)T_\sha(\{1\}^{k-j})\\
=&\,\int_0^1 \om_{-1}^k+ \sum_{j=1}^k  (-1)^j \om_{-1}^{j-1} \om_1 \sha \om_{-1}^{k-j}\\
=&\, \int_0^1 \om_{-1}^k-\om_1\om_{-1}^{k-1}- \om_{-1}( \om_1 \sha \om_{-1}^{k-2})-\om_{-1}^{k-1} \om_1 \\
&\, +\sum_{j=2}^{k-1}  (-1)^j \om_{-1} (\om_{-1}^{j-2}\om_1 \sha \om_{-1}^{k-j})
+ \sum_{j=2}^{k-1} (-1)^j\om_{-1} (\om_{-1}^{j-1}\om_1 \sha \om_{-1}^{k-j-1})  \\
=&\, \int_0^1  \om_{-1}^k-\om_1\om_{-1}^{k-1}
\end{align*}
by changing the index $j\to j+1$ in the first sum since $k$ is odd. Applying the duality relation contained in Theorem~\ref{thm:dualMMVo}, we see that
\begin{equation*}
  T_\sha^\Sy(\{1\}^k)=\int_0^1  \om_0^{k-1}\om_{-1}-\om_0^{k-1}\om_1
  = T(k)-S(k)=-2 \zeta(\ol{k})=2(1-2^{1-k})\zeta(k).
\end{equation*}
Finally, for all $k\ge 2$ we have
\begin{align} \label{equ:SMTdp1}
    T_\sha^\Sy(k) = T(k)+(-1)^k S(k)=&\, \zeta(k)-\zeta(\ol{k})+(-1)^k\big(\zeta(k)+\zeta(\ol{k}) \big) \\
    =&\, \left\{
  \begin{array}{ll}
    2\zeta(k), \hskip1.5cm\ & \hbox{if $k$ is even;}\\
    -2 \zeta(\ol{k}),  & \hbox{if $k\ge 2$ and $k$ is odd}
  \end{array}
\right.         \notag   \\
    \equiv &\, \left\{
  \begin{array}{ll}
      0, \phantom{\zeta(k)} \hskip1.5cm\ & \hbox{if $k$ is even;}\\
     2(1-2^{1-k}) \zeta(k), & \hbox{if $k\ge 2$ and $k$ is odd,}
  \end{array}
\right. \notag
\end{align}
modulo $\zeta(2)\ES_{k-2}^{(1)}$. This completes the proof of the theorem.
\end{proof}

The following combinatorial lemma will be used to prove Theorem~\ref{thm:unitTA}.

\begin{lem}\label{lem:comb}
For all positive integers $n$ and primes $p>n$, we have
\begin{equation*}
\sum_{i=n}^{p-1}(-1)^{i-n}\binom{i-1}{n-1}\equiv 2^{1-n}-1\pmod{p}.
\end{equation*}
\end{lem}

\begin{proof}
Consider the generating function
\begin{align*}
G(x):=\sum_{n=1}^{p-1}\left(\sum_{i=n}^{p-1}(-1)^{i-n}\binom{i-1}{n-{1}}\right)x^{n}
&=\sum_{i=1}^{p-1}\left(\sum_{n=1}^{i}(-1)^{i-n}\binom{i-1}{n-{1}}\right)x^{n}=-\sum_{i=1}^{p-1}x(x-1)^{i-1}.
\end{align*}
By the well-known congruence $\binom{p-1}{k}\equiv (-1)^{k}\pmod{p}$, we obtain
\begin{align*}
G(x)=x\frac{(x-1)^{p-1}-1}{2-x}=&\, \frac{x}{2-x}\sum_{i=1}^{p-1}\binom{p-1}{i}(-1)^{p-i+1}x^{i}\\
 \equiv&\,  -\frac{x}{2-x}\sum_{i=1}^{p-1}x^{i}\equiv-\sum_{j=1}^{\infty}\left(\frac x2\right)^j\sum_{i=1}^{p-1}x^{i} \pmod{p}.
\end{align*}
By Fermat's Little Theorem,
\begin{align*}
G(x) \equiv\sum_{n=2}^{p-1}\left(-\sum_{j=1}^{n-1}2^{-j}\right)x^{n}
+\sum_{n=p}^{\infty}\left(-\sum_{j=1}^{p-1}2^{j}\right)\left(\frac x2\right)^{n}
\equiv\sum_{n=2}^{p-1}(2^{1-n}-1)x^{n} \pmod{p}.
\end{align*}
Comparing the coefficients yields the lemma immediately.
\end{proof}

\begin{thm}\label{thm:unitTA}
Let $k$ be a positive integer. Then
\begin{equation*}
T_\calA(\{1\}^k)=-\zeta_\calA(\ol{k})
=
\left\{
  \begin{array}{ll}
    0, & \hbox{if $k$ is even;} \\
    2\texttt{q}_2, & \hbox{if $k=1$;} \\
    2(1-2^{1-k})\beta_k, & \hbox{if $k$ is odd and $k\ge 3$.}
  \end{array}
\right.
\end{equation*}
\end{thm}

\begin{proof}
If $k$ is even then the theorem follows from \cite[Prop.\ 7]{Zhao2023Dec}. If $k=1$ then
\begin{equation*}
T_\calA(1)=\zeta_\calA(1)-\zeta_\calA(\ol{1})=2\texttt{q}_2
\end{equation*}
by \cite[(8.14)]{Zhao2016}.

For odd $k\ge 3$, we will employ the strategy of expressing a series in two ways, and comparing coefficients. Applying the MZV integral representation, it is easy to see that
\begin{equation}\label{equ:multilog}
\sum_{n_1>\cdots >n_k>0}\frac{x^{n_1}}{n_1\cdots n_k}
=\int_0^{x}\left(\frac{dt}{1-t}\right)^k=\frac{(-1)^k}{k!}\log^{k}\left(1-x\right)
\end{equation}
and
\begin{equation*}
\sum_{\substack{n_1>\cdots >n_k>0\\n_j\equiv k-j+1\pmod 2}}\frac{2^kx^{n_1}}{n_1\cdots n_k}=\int_0^{x}\om_{-1}^k=\frac{(-1)^k}{k!}\log^{k}\left(\frac{1-x}{1+x}\right).
\end{equation*}
According to \eqref{equ:multilog} we can write
\begin{align*}
\sum_{\substack{n_1>\cdots >n_k\\n_j\equiv k-j+1\pmod 2}}\frac{2^kx^{n_1}}{n_1\cdots n_k}
&=\frac{(-1)^{k}}{k!}\log^{k}\left(1-\frac{2x}{1+x}\right)\\
&=\sum_{n_1>\cdots >n_k>0}\frac{1}{n_1\cdots n_{k}}\frac{2^{n_1}x^{n_1}}{(x+1)^{n_1}}\\
&=\sum_{n_1>\cdots n_k}\frac{2^{n_1}}{n_1\cdots n_k}\sum_{m=0}^{\infty}\binom{-n_1}{m}x^{n_1+m} \\
&=\sum_{i=k}^{\infty}\left(\sum_{i\ge n_1>\cdots>n_k>0}\frac{1}{n_1\cdots n_k}2^{n_1}\binom{-n_1}{i-n_1}\right)x^i
\end{align*}
by the index substitution $i=n_1+m$. For any fixed prime $p\ge k+3$, summing the coefficients of $x^i$ for $1\le i\le p-1$ yields
\begin{equation*}
T_\calA(\{1\}^k)_p=\sum_{i=1}^{p-1}\sum_{i\ge n_1>\cdots>n_k>0}\frac{2^{n_1}}{n_1\cdots n_k}\binom{-n_1}{i-n_1}
=\sum_{p-1\ge n_1>\cdots n_k>0}\frac{2^{n_1}}{n_1\cdots n_k}\sum_{i=n_1}^{p-1}(-1)^{i-n_1}\binom{i-1}{n_1-1}.
\end{equation*}
We can now apply Lemma~\ref{lem:comb} to get
\begin{equation*}
T_\calA(\{1\}^k)_p=\sum_{p>n_1>\cdots >n_{k}>0}\frac{2-2^{n_1}}{n_1\cdots n_k}
\equiv \sum_{p>n_1>\cdots >n_{k}>0}\frac{-2^{n_1}}{n_1\cdots n_k} \pmod{p}
\end{equation*}
using the fact that $\zeta_\calA\left(\{1\}^k\right)_p\equiv0 \pmod{p}$ for $p\ge k+3$ (see \cite[Lemma 8.5.3]{Zhao2016}). Finally, taking $t=2$ in \cite[(47)]{SakugawaSeki} implies that
\begin{equation*}
T_\calA(\{1\}^k)_p\equiv \sum_{p>n_1>\cdots >n_{k}>0}\frac{-2^{n_1}}{n_1\cdots n_k}\equiv -\zeta_\calA(\ol{k})_p \pmod{p},
\end{equation*}
hence completing the proof of the theorem by \cite[(8.15)]{Zhao2016}. .
\end{proof}

\begin{cor}\label{cor:unitT-oddWt}
Let $k$ be a positive integer. Then
\begin{equation*}
T_\calA(\{1\}^k)=T_\calA(k)=-\zeta_\calA(\ol{k})=
\left\{
  \begin{array}{ll}
    0, & \hbox{if $k$ is even;} \\
    2\texttt{q}_2, & \hbox{if $k=1$;} \\
    2(1-2^{1-k})\beta_k, & \hbox{if $k$ is odd and $k\ge 3$.}
  \end{array}
\right.
\end{equation*}
\end{cor}
\begin{proof}
This follows immediately from the theorem and \cite[(12)]{Zhao2023Dec}. Note the extra factor of $2^d$ in the definition \eqref{defn:MTV}.
\end{proof}

\begin{rem}\label{rem:unitT-oddWt}
In fact, one can show that for all primes $p$ and positive integers $k<p$ the corresponding $p$-components
\begin{equation}\label{equ:p_compUnitT-oddWt}
T_\calA(\{1\}^k)_p=T_\calA(k)_p.
\end{equation}
The process is a little tedious as one has to go through every step in the proof of Theorem~\ref{thm:unitTA}.
The main idea is that if $k<p-1$ then both (i) $T_\calA\left(\{1\}^k\right)_p=0$ if $2|k$ and (ii) $\zeta_\calA\left(\{1\}^k\right)_p=0$ can be proved using the linear shuffle relation proof as presented in \cite[(12)]{Zhao2023Dec}. The key is  to extract the non-zero $x^p$ term from the product of two power series terms. This is possible as the sum of the smallest powers is at most $k+2\le p$ when $k<p-1$. But this fails if $k=p-1$ which forces us to compute directly and get
\begin{align*}
T_\calA(\{1\}^{p-1})_p=&\, \frac{2^{p-1}}{(p-1)!}\equiv -1 \pmod{p},\\
T_\calA(p-1)_p=&\,\sum_{0< j<p,   j \text{ odd}} \frac{2}{j^{p-1}}\equiv 2\Big(\frac{p-1}{2}\Big)\equiv -1 \pmod{p},
\end{align*}
by Fermat's Little Theorem and Wilson's Theorem. Hence, \eqref{equ:p_compUnitT-oddWt} holds for all $k\le p$.
\end{rem}

\section{Some symmetric multiple $T$-values of height 1}
Recall that the height of $\bfs\in\N^d$ is the number of components of $\bfs$ which are at least 2. The following is a reformulation of \cite[Conj. 2 ]{Zhao2023Dec} concerning indices $\bfs$ of height 1 with leading component equal to 2.

\begin{conj}\label{conj:T21k}
For all $k\in\N$, we have
\begin{equation*}
  T_\calA(2,\{1\}^k)=(-1)^{k}  T_\calA(1,k+1), \qquad
  T_\sha^\Sy(2,\{1\}^k)\equiv (-1)^{k} T_\sha^\Sy(1,k+1) \pmod {\zeta(2)\ES_{k}}.
\end{equation*}
\end{conj}

\begin{rem}
There was an obvious typo on the left-hand side of the second equation in \cite[Conj. 2 ]{Zhao2023Dec}.
\end{rem}

We now provide a stronger and more precise result than the second equation in Conjecture \ref{conj:T21k}. We will prove the first relation in Corollary~\ref{cor:TA21_k}.

\begin{thm}\label{thm:SMT21k}
For all $k\in\N$, we have
\begin{equation*}
  T_\sha^\Sy(2,\{1\}^k)-(-1)^{k} T_\sha^\Sy(1,k+1)=
  \left\{
    \begin{array}{ll}
      \displaystyle 2\sum_{n=1}^{\frac{k-1}{2}}\zeta(2n)T(k+2-2n), & \hbox{if $k$ is odd;} \\
      (k+2)T(k+2), \phantom{\frac12} & \hbox{if $k$ is even.}
    \end{array}
  \right.
\end{equation*}
In particular,
\begin{equation*}
  T_\sha^\Sy(2,\{1\}^k)\equiv (-1)^{k} T_\sha^\Sy(1,k+1) \pmod {\zeta(2)\ES_{k}}.
\end{equation*}
\end{thm}

\begin{proof}
We proceed by considering two different cases according to the parity of $k$.

\medskip
\noindent
\textbf{Case (i). $k$ even.} By definition
\begin{equation*}
T_\sha^\Sy(2,\{1\}^k)=T(2,\{1\}^k)+\sum_{j=0}^k (-1)^j S_\sha(\{1\}^j,2)T_\sha(\{1\}^{k-j}).
\end{equation*}
By the duality relation \cite[Thm. 3.1]{KanekoTs2020} we see that $T(2,\{1\}^k)=T(k+2)$. Hence,
\begin{align}
T_\sha^\Sy&\, (2,\{1\}^k)-T(k+2)=\sum_{j=0}^k (-1)^j  \om_{-1}^j \om_0 \om_1 \sha \om_{-1}^{k-j}    \notag \\
=&\,  \om_0( \om_1 \sha \om_{-1}^{k})+\om_{-1}(\om_0 \om_1 \sha \om_{-1}^{k-1})    \notag\\
&  +\sum_{j=1}^{k-1} (-1)^j  \big[\om_{-1}(\om_{-1}^j \om_0 \om_1 \sha \om_{-1}^{k-j-1}) + \om_{-1}(\om_{-1}^{j-1} \om_0 \om_1 \sha \om_{-1}^{k-j})\big]  +(-1)^k \om_{-1}^k \om_0 \om_1    \notag\\
=&\,  \om_0( \om_1 \sha \om_{-1}^k)+ \sum_{j=0}^{k-1}(-1)^j \om_{-1}(\om_{-1}^j \om_0 \om_1 \sha \om_{-1}^{k-j-1} )
+\sum_{j=1}^{k} (-1)^j \om_{-1}(\om_{-1}^{j-1} \om_0 \om_1 \sha \om_{-1}^{k-j})   \label{equ:cancelSigma}\\
=&\,  \om_0( \om_1 \sha \om_{-1}^k )    \notag
\end{align}
by changing the index $j\to j+1$ in the last sum of \eqref{equ:cancelSigma}. Now, the duality relation of MMVs given by
Thm. \ref{thm:dualMMVo} yields
\begin{align} \label{equ:LHSafterDual}
T_\sha^\Sy (2,\{1\}^k)-T(k+2)=&\, \big[\om_0^k \sha (\om_0+\om_1-\om_{-1})\big] \om_{-1}.
\end{align}
On the other hand,
\begin{equation}\label{equ:RHS}
    T_\sha^\Sy(1,k+1)= T_\sha(1,k+1)-T_\sha(1)T(k+1)+T(k+1,1)
 =\int_0^1 \om_{-1}\om_0^{k}\om_{-1}-\om_{-1} \sha\om_0^{k}\om_{-1} + \om_0^{k}\om_{-1}\om_{-1}
\end{equation}
Combining \eqref{equ:LHSafterDual} and \eqref{equ:RHS}, we get
\begin{align}
&\, T_\sha^\Sy (2,\{1\}^k)-T(k+2)- T_\sha^\Sy(1,k+1)\notag\\
=&\,\int_0^1 (k+1) \om_0^{k+1}\om_{-1}+[(\om_1-\om_{-1}) \sha \om_0^k]\om_{-1} - \om_{-1}\om_0^{k}\om_{-1}+\om_{-1} \sha\om_0^{k}\om_{-1} -\om_0^{k}\om_{-1}\om_{-1} \notag\\
=&\,\int_0^1 (k+1) \om_0^{k+1}\om_{-1}+\om_1 \sha \om_0^k\om_{-1}- \om_{-1}\om_0^k\om_{-1} - \om_0^k\om_{-1}\om_1.\label{equ:LHSafterDual-RHS}
\end{align}
Setting
\begin{equation*}
  a=\frac{dt}{t},\quad b=\frac{dt}{1-t},\quad c=\frac{dt}{-1-t},
\end{equation*}
we see that $\om_1=b+c$ and $\om_{-1}=b-c$. Thus, \eqref{equ:LHSafterDual-RHS} becomes
\begin{align*}
&\,T_\sha^\Sy (2,\{1\}^k)- T_\sha^\Sy(1,k+1)-T(k+2) \\
=&\,\int_0^1  (k+1) \om_0^{k+1}\om_{-1}+(b+c) \sha a^k(b-c)-(b-c)a^k(b-c) - a^k(b-c) (b+c) \\
=&\,\int_0^1 (k+1) \om_0^{k+1}\om_{-1}+c \sha a^k(b-c)+a\big[b \sha a^{k-1} (b-c)\big] +c a^k(b-c) - a^k(b-c) (b+c) \\
=&\,\int_0^1 (k+1) \om_0^{k+1}\om_{-1}+\sum_{j=0}^k a^j ca^{k-j}b+a^kbc -\sum_{j=0}^k a^j ca^{k-j}c-a^kc^2\\
&+\sum_{j=1}^k a^j b a^{k-j} b+a^kb^2-\sum_{j=1}^k a^j b a^{k-j} c-a^kcb  + c a^k(b-c) - a^k(b-c) (b+c).
\end{align*}
Therefore,
\begin{align*}
& T_\sha^\Sy(2,\{1\}^k)- T_\sha^\Sy(1,k+1) \\
=&\, (k+2) T(k+2)+\sum_{j=0}^k \zeta(\ol{j+1},\ol{k-j+1}) + \zeta(k+1,\ol{1})
-\sum_{j=0}^k \zeta(\ol{j+1},k-j+1) - \zeta(\ol{k+1},1) \\
& +\sum_{j=1}^k \zeta(j+1,k-j+1) + \zeta(k+1,1)
-\sum_{j=1}^k \zeta(j+1,\ol{k-j+1}) - \zeta(\ol{k+1},\ol{1}) \\
&+\zeta( \ol{1},\ol{k+1})-\zeta(\ol{1},k+1) -\zeta(k+1,1)-\zeta(k+1,\ol{1})+\zeta(\ol{k+1},1)+\zeta(\ol{k+1},\ol{1}).
\end{align*}
The four sum formulas in \cite[Thm. 4.2]{BCJXXZhao2020c} can reduce the above to the following very simple final form
\begin{align*}
T_\sha^\Sy(2,\{1\}^k)- T_\sha^\Sy(1,k+1) =&\, (k+2) T(k+2).
\end{align*}

Since $k$ is even, we get from the definition
\begin{equation*}
T(k+2)=\zeta(k+2)-\zeta(\ol{k+2})=(2-2^{-k-1})\zeta(k+2)\in \zeta(2) \ES_{k}.
\end{equation*}
This completes the proof of Thm. \ref{thm:SMT21k} for even $k$ which also implies the even $k$ case of the second relation in Conjecture~\ref{conj:T21k}.

\medskip
\noindent
\textbf{Case (ii). $k$ odd.} By definition, the left-hand side of the second relation in Conjecture \ref{conj:T21k} is
\begin{equation*}
T_\sha^\Sy(2,\{1\}^k)=T(k+2)+\sum_{i=0}^{k}(-1)^i T_\shuffle(\{1\}^i,2)T_\shuffle(\{1\}^{k-i}).
\end{equation*}
On the word level, the sum of products in the above is given by
\begin{align*}
& \sum_{i=0}^{k}(-1)^i  \om_{-1}^i \om_0 \om_{-1} \sha \om_{-1}^{k-i} \\
=&\, \om_0  \om_{-1}  \sha \om_{-1}^{k} +  \om_{-1}^k \om_0 \om_{-1}
+ \sum_{i=1}^{k-1}(-1)^i \om_{-1} (\om_{-1}^{i-1} \om_0 \om_{-1} \sha \om_{-1}^{k-i} )
+  \sum_{i=1}^{k-1}(-1)^i \om_{-1} (\om_{-1}^{i} \om_0 \om_{-1} \sha \om_{-1}^{k-i-1} ) \\
=&\, \om_0  (\om_{-1}  \sha \om_{-1}^{k})
+ \sum_{i=1}^{k}(-1)^i \om_{-1} (\om_{-1}^{i-1} \om_0 \om_{-1} \sha \om_{-1}^{k-i} )
+ \sum_{i=0}^{k-1}(-1)^i \om_{-1} (\om_{-1}^{i} \om_0 \om_{-1} \sha \om_{-1}^{k-i-1} ) \\
=&\, (k+1)\om_0 \om_{-1}^{k+1}
\end{align*}
by changing the index $i\to i-1$ in the last sum. Hence,
\begin{equation}\label{equ:TS21k}
T^{\mathcal S}(2,\{1\}^k)=(k+2) T(2,\{1\}^k) =(k+2) T(k+2)
\end{equation}
by the duality relation.

Next, by definition we can express MTVs on the right-hand side of the second relation in Conjecture \ref{conj:T21k} as
\begin{equation*}
T_\sha^\Sy(1,k+1)=T_\shuffle(1,k+1)-T_\shuffle(1)T(k+1)-T(k+1,1).
\end{equation*}
By the shuffle product,
\begin{equation*}
T_\shuffle(1)T(k+1)=\int_0^1 \om_{-1}\sha \om_0^k \om_{-1} =\int_0^1 \om_0^{k}\om_{-1}^2+\sum_{i=0}^{k}\om_0^i\om_{-1}\om_0^{k-i}\om_{-1}.
\end{equation*}
Thus, simplifying and using \cite[Cor. 4.3]{BCJXXZhao2020c} by taking the same notation $v=k+1$ there, we obtain
\begin{align}
T^{\mathcal S}(1,v)=&-2T(v,1)-\sum_{i=1}^{k}T(1+i,v-i)  \notag\\
=&-2T(v,1)-2T(k+2) - 2\zeta(\ol{v},\bar 1) - 2\zeta(\bar1,\ol{v}) + 2\zeta(\ol{v},1)+2\zeta(\ol{1},v) \notag\\
=&-2T(k+2)+2\zeta(v,\bar 1) - 2\zeta(\bar1,\ol{v})-2\zeta(v,1)+2\zeta(\ol{1},v).\label{equ:T1v}
\end{align}

Let $w=k+2$. For MZVs, we can use the stuffle relations and the sum formula to get
\begin{equation}\label{equ:noBars}
\sum_{n=1}^{\frac{k-1}{2}}\zeta(2n)\zeta(w-2n) = \sum_{a,b\ge 2}\zeta(a,b)+\frac{k-1}{2}\zeta(w)=\frac{k+1}{2}\zeta(w)-\zeta(k+1,1)
\end{equation}
by \cite[Thm. 1]{GanglKaZa2006}. For the alternating sums, we have
\begin{equation}\label{equ:twoBars1}
\sum_{n=1}^{\frac{k-1}{2}}\zeta(2n)\zeta(\overline{w-2n})=\sum_{n=1}^{\frac{k-1}{2}} \zeta(2n,\overline{w-2n})+\sum_{n=1}^{\frac{k-1}{2}}\zeta(\overline{w-2n},2n)+\frac{k-1}{2}\zeta(\overline{w}).
\end{equation}
By the middle equation in \cite[Prop. 4.1]{BCJXXZhao2020c}, we see that
\begin{align*}
&\,\sum_{a+b=w}(x+y)^{a-1}(\zeta(a,\bar b)y^{b-1}+\zeta(\bar a, \bar b)x^{b-1})\\
=&\, \sum_{a+b=w}\zeta(a,\bar b)x^{a-1}y^{b-1}+\zeta(\bar a,b)y^{a-1} x^{b-1}+\frac{x^{w-1}-y^{w-1}}{x-y}\zeta(\bar w).
\end{align*}
Substituting $(x,y)=(1,0)$ and $(x,y)=(0,1)$ yields
\begin{equation}\label{equ:sumForm1}
\sum_{\substack{a+b=w \\a,b\ge 2}}\Big[\zeta(\bar a, b)+\zeta(a,\bar b)\Big]=-\zeta(\bar 1, v)+\zeta(\bar 1,\overline{v})+\zeta(w)+\zeta(\bar w)-\zeta(\overline{v},1)-\zeta(v,\bar 1)
\end{equation}
and substituting $(x,y)=(-1,1)$ yields
\begin{equation}\label{equ:sumForm2}
\sum_{\substack{a+b=w \\ a,b\ge 2}}\Big[(-1)^a\zeta(\bar a, b)+(-1)^b\zeta(a,\bar b)\Big]=-\zeta(\bar 1, \overline{v})+\zeta(\bar 1, v)+\zeta(v,\bar 1)-\zeta(\overline{v}, 1).
\end{equation}
Taking the difference between \eqref{equ:sumForm1} and \eqref{equ:sumForm2}, we get
\begin{equation}\label{equ:twoBars}
    \sum_{\substack{a+b=w \\ a,b\ge 2\\ a\text{ odd}}} \Big[\zeta(\bar a, b)+\zeta(b,\bar a)\Big] = -\zeta(\bar 1,v)+\zeta(\bar 1,\overline{v})-\zeta(v,\bar 1)+\frac{\zeta(w)+\zeta(\bar w)}{2}.
\end{equation}
Combining this with \eqref{equ:noBars} and \eqref{equ:twoBars1}, we finally arrive at
\begin{align*}
2\sum_{n=1}^{\frac{k-1}{2}}\zeta(2n)T(w-2n)
=&\, kT(w)-2\zeta(v,1)+2\zeta(v,\bar 1)-2\zeta(\bar 1,\overline{v})+2\zeta(\bar 1,v) \\
=&T_\sha^\Sy(2,\{1\}^k)+T_\sha^\Sy(1,k+1)
\end{align*}
by \eqref{equ:TS21k} and \eqref{equ:T1v}.
This completes the proof of Theorem \ref{thm:SMT21k} when $k$ is odd.

Furthermore, since
\begin{align*}
T(m)=\zeta(m)-\zeta(\ol{m})=(2-2^{1-m})\zeta(m)
\end{align*}
we see that
\begin{equation*}
\zeta(2n)T(k+2-2n)\in \zeta(2)\ES_{k}^{(1)} \subset \zeta(2)\ES_{k}
\end{equation*}
for all $n=1,\dots,(k-1)/2.$ This shows that the second relation in Conjecture \ref{conj:T21k} holds for odd $k$.
\end{proof}

In \cite[Prop. 3]{Zhao2023Dec}, the second author proved the following result.

\begin{prop}\label{prop:FMT1ell2lk} \emph{(\cite[Prop. 3]{Zhao2023Dec})}
If $k$ is odd, then for all $\ell \le k$, we have
\begin{equation}\label{equ:FMT121}
T_\calA(\{1\}^\ell,2,\{1\}^{k-\ell})=\frac{(-1)^{\ell}}{\ell+1} \binom{k+1}{\ell}  T_\calA(2,\{1\}^k).
\end{equation}
\end{prop}

Following the philosophy of Conjecture~\ref{conj:KanekoZagierMTV} we now present its symmetric MTV analog as follows.

\begin{thm}  \label{thm:SMT121}
If $k\in\N$ is odd, then for all $0\le \ell \le k$ we have
\begin{equation}\label{equ:SMT121}
T_\sha^\Sy(\{1\}^\ell,2,\{1\}^{k-\ell})=(-1)^{\ell}\binom{k+2}{\ell+1} T(k+2)
    =\frac{(-1)^{\ell}}{\ell+1} \binom{k+1}{\ell}  T_\sha^\Sy(2,\{1\}^k).
\end{equation}
\end{thm}

\begin{proof}
To prove the first equation, let $n=k-\ell$. Then, by definition
\begin{align}
T_\sha^\Sy(\{1\}^\ell,2,\{1\}^n)=&\, \sum_{j=0}^\ell (-1)^j T_\sha(\{1\}^j) T_\sha(\{1\}^{\ell-j},2,\{1\}^n) \notag\\
&\, +\sum_{j=0}^n (-1)^{n-j+\ell} T_\sha(\{1\}^{n-j},2,\{1\}^\ell) T_\sha(\{1\}^j) \notag\\
=&\,\sigma(\ell,n)-\sigma(n,\ell)   \label{equ:ht1}
\end{align}
since $n+\ell=k$ is odd, where $\sigma(\ell,n)$ denotes the first sigma sum. Thus,
\begin{align*}
\sigma(\ell,n)= &\,\int_0^1 \sum_{j=0}^\ell (-1)^j   \om_{-1}^j \sha \om_{-1}^{\ell-j} \om_0\om_{-1}^{n+1}   \\
=&\,\int_0^1 \om_{-1}^{\ell} \om_0\om_{-1}^{n+1} + (-1)^\ell   \om_{-1}( \om_{-1}^{\ell-1}  \sha  \om_0\om_{-1}^{n+1})
 + (-1)^\ell   \om_0( \om_{-1}^\ell \sha  \om_{-1}^{n+1}) \\
&\,  +\sum_{j=1}^{\ell-1}  (-1)^j     \om_{-1}(\om_{-1}^{j-1} \sha \om_{-1}^{\ell-j} \om_0\om_{-1}^{n+1} )
+\sum_{j=1}^{\ell-1}  (-1)^j     \om_{-1}(\om_{-1}^j \sha \om_{-1}^{\ell-j-1} \om_0\om_{-1}^{n+1} )\\
=&\,   (-1)^\ell \int_0^1   \om_0( \om_{-1}^\ell \sha  \om_{-1}^{n+1}) \qquad(\text{by $j\to j-1$ in the last sigma sum})\\
=&\,   (-1)^\ell \int_0^1  \binom{\ell+n+1}{\ell} \om_0\om_{-1}^{\ell+n+1}\\
=&\,  (-1)^\ell  \binom{\ell+n+1}{\ell}  T(k+2).
\end{align*}
By \eqref{equ:ht1}, since $\ell+n=k$ is odd we get
\begin{align*}
T_\sha^\Sy(\{1\}^\ell,2,\{1\}^n)=&\,\left( (-1)^\ell   \binom{\ell+n+1}{\ell} - (-1)^n   \binom{\ell+n+1}{n}\right)T(k+2) \\
=&\,(-1)^\ell \left( \binom{k+1}{\ell}+\binom{k+1}{\ell+1}\right)T(k+2)    \\
=&\,(-1)^{\ell}\binom{k+2}{\ell+1}  T(k+2).
\end{align*}
This shows the first equation in \eqref{equ:SMT121}. It is clear that the second equation follows immediately from the first by taking $\ell=0$. This concludes the proof of the theorem.
\end{proof}

\begin{cor}\label{cor:SMTV21k}
If $k\in\N$ is odd, then we have
\begin{equation}\label{equ:SMTV21k}
T_\sha^\Sy(2,\{1\}^k)=\frac{(k+2)(2^{k+2}-1)}{2(2^{k+1}-1)}T_\sha^\Sy(k+2).
\end{equation}
\end{cor}

\begin{proof} Since
\begin{align*}
  T(k+2) =&\, \zeta(k+2) - \zeta(\ol{k+2})=(2-2^{-1-k}) \zeta(k+2),\\
T_\sha^\Sy(k+2) =&\, T(k+2)-S(k+2)=-2\zeta(\ol{k+2})=2(1-2^{-1-k}) \zeta(k+2).
\end{align*}
Hence, \eqref{equ:SMTV21k} follows from Theorem~\ref{thm:SMT121} immediately.
\end{proof}

\section{Duality relations among finite and symmetric MTVs}
In this section, we  present some duality type results on both finite and symmetric MTVs involving height one indices. Recall that the star values $\zeta^\star$ are defined by allowing equalities among the summation indices in \eqref{defn:MZV}. One can similarly define the star versions of the finite MZVs for which Hoffman (see \cite{Hoffman2004} and \cite[Thm. 8.5.11]{Zhao2016}) discovered that, for any $\bfs\in\N^d$,
\begin{equation}\label{equ:FMZVv-dual}
\zeta_\calA^\star(\bfs)=-\zeta_\calA^\star(\bfs^\vee),
\end{equation}
where $\bfs^\vee$ is defined by swapping the plus signs ``$+$'' and the commas ``,'' in
\begin{equation*}
\bfs=(\underbrace{1+\cdots+1}_{\text{$s_1$ times}},\cdots,\underbrace{1+\cdots+1}_{\text{$s_d$ times}}).
\end{equation*}
For example, $(2,1,1)^\vee=(1,3)$. Moreover, Jarossay showed that the symmetric MZVs also satisfy the same type of duality relation similar to \eqref{equ:FMZVv-dual} modulo $\zeta(2)$ (see  \cite[Corollary 1.12]{Jarossay2014} or \cite[Theorem 6.3.5]{Zhao2016}) .

Even though MTVs satisfy exactly the same duality relations satisfied by the MZVs, the finite and symmetric MTVs in general do not possess the same type of duality as that appearing in \eqref{equ:FMZVv-dual}. For example, for $p=101$,
\begin{equation*}
T_\calA^\star(2,1,1)_p\equiv 57, \quad T_\calA^\star(1,3)_p\equiv 22  \pmod{p}.
\end{equation*}

However, Theorem~\ref{thm:unitSMT}, Corollary~\ref{cor:unitT-oddWt}, Theorem~\ref{thm:SMT21k}, and Corollary~\ref{cor:TA21_k} state that
\begin{align*}
T_\calA(\bfs)=&\, \delta(\bfs) T_\calA(\bfs^\vee),\quad  T_\sha^\Sy(\bfs)\equiv \delta(\bfs)T_\sha^\Sy(\bfs^\vee) \pmod{\zeta(2)\ES_{k-2}^{(1)}}
\end{align*}
where
\begin{enumerate}
    \item [\upshape{(1)}] $\delta(\bfs)=1$ for $\bfs=(\{1\}^k)$
    \item [\upshape{(2)}] $\delta(\bfs)=(-1)^k$ for $\bfs=(2,\{1\}^k)$, where $n,k\in\N$;
\end{enumerate}

Further, we have the following duality type results.

\begin{thm}\label{thm:dblTs}
Suppose $a,b\in\N$ and $w=a+b$ is odd. Then
\begin{align*}
S_\calA(a,b)=T_\calA(a,b)=&\, (-1)^a \binom{w}{a} (2-2^{1-w}) \beta_w =-T_\calA((a,b)^\vee), \\
S_\sha^\Sy(a,b)=T_\sha^\Sy(a,b)\equiv&\,  (-1)^{a}\binom{w}{a}(2-2^{1-w}) \zeta(w)\equiv -T_\sha^\Sy((a,b)^\vee) \pmod{\zeta(2)\ES_{k-2}^{(1)}}.
\end{align*}
\end{thm}
\begin{proof}
We observe that for any $a,b\in\N$,
\begin{equation*}
(a,b)^\vee=(\{1\}^{a-1},2,\{1\}^{b-1}).
\end{equation*}

The first evaluation of double finite $S$- and $T$-values in the theorem is given by \cite[(42)]{Zhao2024a} (notice again that the extra factor of $2^d$ in the definition of MTVs in this paper). The second follows from Corollary~\ref{cor:TA21_k}.

Turning to the symmetric MTVs, from the proof of \cite[Cor. 4.2]{XuZhao2023Oct} (note that the definite
of MTVs there is opposite to that in this paper), we see that modulo $\zeta(2)\ES_{w-2}$
\begin{align*}
T_\sha^\Sy(a,b)=&\, T_\sha(a,b)+(-1)^{a}T_\sha(a)T_\sha(b)-T_\sha(b,a) \\
\equiv &\, \left( (-1)^{a}\binom{w-1}{b-1} -(-1)^{b}\binom{w-1}{a-1}\right) T(w) \\
\equiv  &\, (-1)^{a}\binom{w}{a}T(w)  \\
\equiv  &\, (-1)^{a}\binom{w}{a}(2-2^{1-w}) \zeta(w).
\end{align*}
Similarly,
\begin{align*}
S_\sha^\Sy(a,b)=&\, S_\sha(a,b)+(-1)^{a}S_\sha(a)S_\sha(b)-S_\sha(b,a) \\
\equiv &\, \left( (-1)^{a}\binom{w-1}{b-1} -(-1)^{b}\binom{w-1}{a-1}\right) T(w) \\
\equiv  &\, (-1)^{a}\binom{w}{a}T(w)  \\
\equiv  &\, (-1)^{a}\binom{w}{a}(2-2^{1-w}) \zeta(w).
\end{align*}
Comparing with \eqref{equ:SMT121}, we obtain the theorem immediately.
\end{proof}

Our next result shows that sometimes a duality holds only after taking a reversal. We will need the following two lemmas.

\begin{lem}\label{lem:T(x)_integral} \emph{(\cite[Line 8]{Zhao2023Dec})}
For any positive integers $s_1,\dots,s_d$ and $|x|<1$, we have
\begin{equation*}
\sum_{\substack{n_1>\cdots>n_d>0\\n_j\equiv d-j+1}}\frac{2^dx^{n_1}}{n_1^{s_1}\cdots n_d^{s_d}}
=\int_0^{x}\om_0^{s_1-1}\om_{-1}\cdots \om_0^{s_d-1}\om_{-1}.
\end{equation*}
\end{lem}

\begin{lem}\label{lem:F(x)}
For any index $\bfs=(s_1,\dots,s_d)\in\N^d$ with $s_d\ge 2$, define $\bfs-1:=(s_1,\dots,s_d-1).$
Then, for any prime $p$, we have
\begin{align}\label{equ:keyReduction}
T_\calA(\bfs,\{1\}^k)_p &\equiv -\sum_{j=0}^{p-k-2} \binom{p-k-1}{j}  2^jB_j\Big(\frac12\Big) \frac{T_\calA(\bfs-1,\{1\}^{k+j})_p }{p-k-1}  \pmod{p}.
\end{align}
\end{lem}

\begin{proof}
Consider the generating function
\begin{equation*}
F_k(\bfs;x):=\sum_{\substack{n_1>\cdots >n_{d+k}>0\\ n_j\equiv d+k+1-j\pmod 2}}\frac{x^{n_1}}{n_1^{s_1}\cdots n_d^{s_d}n_{d+1}\cdots n_{d+k}}
\end{equation*}
We recall that the duality of MTVs is proved by a change of variable $t\leftrightarrow  (1-t)/(1+t)$ under which $\om_0 \leftrightarrow -\om_{-1}$. Then, all the signs cancel by reversing the path.
Hence, applying Lemma \ref{lem:T(x)_integral} and the duality trick, we find
\begin{alignat*}{4}
F_k(\bfs;x)&=\int_0^{x}\om_0^{s_1-1}\om_{-1}\cdots\om_0^{s_d-1}\om_{-1}^{k+1}\\
&=\frac{1}{(k+1)!}\int_0^x\om_0^{s_1-1}\om_{-1}\cdots\om_0^{s_d-1}\left(\int_0^{t} \om_{-1}\right)^{k+1} \quad &&\text{(shuffle product)}\\
&=\frac{(-1)^{k+1}}{(k+1)!}\int_0^x\om_0^{s_1-1}\om_{-1}\cdots\om_0^{s_d-1}\Bigg\{\frac{1}{t}\log^{k+1}\left(\frac{1-t}{1+t}\right)dt\Bigg\}\\
&=\frac{(-1)^{k+1}}{(k+1)!}\int_{\frac{1-x}{1+x}}^{1}\Bigg\{\frac{2}{1-t^2}\log^{k+1}\left(t\right)dt\Bigg\}\om_{-1}^{s_d-2} \om_0\cdots\om_0\om_{-1}^{s_1-1} \quad&&\text{(duality)}\\
&=\frac{(-1)^{k+1}}{(k+1)!}\int_{\log\left(\frac{1-x}{1+x}\right)}^{0}
\Bigg\{\frac{2t^{k+1}e^tdt}{1-e^{2t}}\Bigg\}\tom_{-1}^{s_d-2} \tom_0\cdots\tom_0\tom_{-1}^{s_1-1}  \quad &&(\text{by }t\to e^t) \\
&=\frac{(-1)^{k}}{(k+1)!}\int_{\log\left(\frac{1-x}{1+x}\right)}^{0}
\left(\sum_{j=0}^{\infty}\frac{2^jB_j\Big(\frac12\Big)}{j!}t^{j+k} dt\right) \tom_{-1}^{s_d-2} \tom_0\cdots\tom_0\tom_{-1}^{s_1-1} &&
\end{alignat*}
by using the generating function of Bernoulli polynomials at $1/2$. Here, $\tom_\ga=\om_\ga(e^t)$ for $\ga=0,-1.$ Moreover, it is easy to check that when $|x|<1-1/e$ the power series converges uniformly. Reversing the path of the iterated integral, we obtain
\begin{alignat*}{4}
F_k(\bfs;x)&=\frac{(-1)^{k+|\bfs|-1}}{(k+1)!}\sum_{j=0}^{\infty}\frac{2^jB_j\Big(\frac12\Big)}{j!} \int_0^{\log\left(\frac{1-x}{1+x}\right)}\tom_{-1}^{s_1-1}\tom_0\cdots
 \tom_0\tom_{-1}^{s_d-2}\left(\int_0^t dt\right)^{j+k} \,dt\\
&=\frac{(-1)^{k+|\bfs|-1}}{(k+1)!}\sum_{j=0}^{\infty}\frac{2^jB_j\Big(\frac12\Big)}{j!}(j+k)!
\int_0^{\log\left(\frac{1-x}{1+x}\right)}
 \tom_{-1}^{s_1-1}\tom_0\cdots
 \tom_0\tom_{-1}^{s_d-2} (dt)^{j+k+1} \quad&&\ \text{(shuffle product)}\\
&=\frac{1}{(k+1)!}\sum_{j=0}^{\infty}\frac{(-2)^jB_j\Big(\frac12\Big)}{j!}(j+k)!
    \int_{\frac{1-x}{1+x}}^{1} \om_0^{j+k+1}\om_{-1}^{s_d-2} \om_0\cdots\om_0\om_{-1}^{s_1-1} \quad&&\ \text{(path reversal)}\\
&=\frac{1}{(k+1)!}\sum_{j=0}^{\infty}\frac{(-2)^jB_j\Big(\frac12\Big)}{j!}(j+k)!
    \int_0^{x}\om_0^{s_1-1}\om_{-1}\cdots\om_0^{s_d-2}\om_{-1}^{j+k+1}\quad&&\ \text{(duality)}\\
&=\sum_{j=0}^{\infty} 2^jB_j\Big(\frac12\Big) \frac{1}{k+j+1} \binom{k+j+1}{j} F_{j+k}(\bfs-1;x),
\end{alignat*}
where we have removed the sign $(-1)^j$ since $B_j(1/2)=(2^{1-j}-1)B_j=0$ when $j$ is odd.
For any prime $p$, taking the sum of coefficients of $x^j$ for all $0<j<p$ on both sides of the above, we find
\begin{align}\label{equ:FMTVbfs1k}
&T_\calA(\bfs,\{1\}^k)_p  \equiv \sum_{j=0}^{p-k-d-1}2^jB_j\Big(\frac12\Big) \frac{1}{k+j+1} \binom{k+j+1}{j} T_\calA(\bfs-1,\{1\}^{j+k})_p    \pmod{p}
\end{align}
since the smallest nontrivial power in $F_{\bfs-1,j+k}(x)$ is $j+k+d$. Further, we have
\begin{align*}
\frac{-1}{p-k-1} \binom{p-k-1}{j}
&=\frac{1}{k+1} \frac{(p-k-1)\cdots(p-k-j)}{j!} \\
&\equiv\frac{(-1)^{j}}{k+1} \frac{(k+1)\cdots(k+j)}{j!}
\equiv\frac{(-1)^{j}}{k+j+1} \binom{k+j+1}{j} \pmod{p},
\end{align*}
so that \eqref{equ:FMTVbfs1k} implies \eqref{equ:keyReduction} immediately and the lemma follows.
\end{proof}

\begin{lem}\label{lem:S(x)}
For any index $\bfs=(s_1,\dots,s_d)\in\N^d$ and $k\in\N$, define $\bfs+k=(s_1,\dots,s_{d-1},s_d+k).$
Then, for any prime $p> k+1$, we have
\begin{align}\label{equ:T2S}
T_\calA(\bfs,k+1)_p\equiv \sum_{j=0}^{p-k-2} \binom{p-k-1}{j}2^jB_j\Big(\frac12\Big)\frac{S_\calA(\bfs+k+j)_p}{p-k-1}.
\end{align}
\end{lem}

\begin{proof}
By the Seki-Bernoulli formula,
\begin{equation*}
S_k(r):=\sum_{j=1}^{r}j^k=\frac{1}{k+1}\sum_{j=0}^{k}(-1)^j\binom{k+1}{j}B_j r^{k+1-j}.
\end{equation*}
For all even $r\in\N$, noticing that $p>k+1$, we get
\begin{align*}
\sum_{m=1}^{r} \frac{1-(-1)^m}{m^{k+1}}&\equiv\sum_{m=1}^{r}(1-(-1)^m)m^{p-k-2}\equiv2\sum_{1\le m\le r, m \text{ odd}}  m^{p-k-2} \pmod{p}\\
&\equiv 2S_{p-k-2}(r)-2^{p-k-1}S_{p-k-2}\left( \frac{r}{2} \right)\\
&\equiv \frac{2}{p-k-1}\sum_{j=0}^{p-k-2}(-1)^j \binom{p-k-1}{j}\left(r^{p-k-1-j}-2^{p-k-2}\left(\frac{r}{2}\right)^{p-k-1-j}\right)B_{j}\\
&\equiv \frac{1}{p-k-1}\sum_{j=0}^{p-k-2} \binom{p-k-1}{j}2^jB_j\Big(\frac12\Big)r^{p-k-1-j}
\end{align*}
since $B_j(1/2)=(2^{1-j}-1)B_{j}=0$ for all odd $j$. Hence, modulo $p$
\begin{align*}
T_\calA& (\bfs,k+1)_p\equiv\sum_{\substack{p>r_1>\cdots >r_{d+1}>0 \\ r_h\equiv d-h\pmod 2}} \frac{2^{d+1}}{r_1^{s_1}\cdots r_d^{s_d} r_{d+1}^{k+1}} \notag \\
&\equiv\sum_{\substack{p>r_1>\cdots >r_d>0\\ r_h\equiv d-h\pmod 2}} \frac{2^d}{r_1^{s_1}\cdots r_d^{s_d}}\sum_{r_{d+1}=1}^{r_d}\frac{1-(-1)^{r_{d+1}}}{r_{d+1}^{k+1}}\notag\\
&\equiv\frac{1}{p-k-1}\sum_{\substack{p>r_1>\cdots >r_d>0\\r_h\equiv d-h\pmod 2}}\frac{2^d}{r_1^{s_1}\cdots r_d^{s_d}}\sum_{j=0}^{p-k-2} \binom{p-k-1}{j}2^jB_j\Big(\frac12\Big)r_d^{p-k-1-j}\notag\\
&\equiv \frac{1}{p-k-1}\sum_{j=0}^{p-k-2} \binom{p-k-1}{j}2^jB_j\Big(\frac12\Big)\sum_{\substack{p>r_1>\cdots >r_d>0\\r_h\equiv d-h\pmod 2}}\frac{2^d}{r_1^{s_1}\cdots r_d^{s_d+k+j}}\notag\\
&\equiv \frac{1}{p-k-1}\sum_{j=0}^{p-k-2} \binom{p-k-1}{j}2^jB_j\Big(\frac12\Big)S_\calA(\bfs+k+j)_p.
\end{align*}
This completes the proof of the lemma.
\end{proof}

\begin{lem}\label{lem:S2(x)}
For any index $\bfs=(s_1,\dots,s_d)\in\N^d$ and $k\in\N$, define $\bfs+k=(s_1,\dots,s_{d-1},s_d+k).$
Then, for any prime $p>k+1$, we have
\begin{equation*}
S_\calA(\bfs,k+1)_p\equiv -T_\calA(k+1)_p \cdot T_\calA(\bfs)_p
+\sum_{j=0}^{p-k-2}\binom{p-k-1}{j} 2^jB_j\Big(\frac12\Big)\frac{T_{\calA}(\bfs+k+j)_p }{p-k-1}\pmod{p}.
\end{equation*}
\end{lem}

\begin{proof}
Set $N=p-k-1\ge 1$ for simplicity. Then
\begin{align*}
\ S_\calA(\mathbf s,k+1)_p&
\equiv\sum_{\substack{p>r_1>\cdots >r_{d+1}>0 \\ r_{j}\equiv d+1-j\pmod 2}} \frac{2^{d+1}}{r_1^{s_1}\cdots r_{d}^{s_d}r_{d+1}^{k+1}} \\
&\equiv\sum_{\substack{p>r_1>\cdots >r_{d}>0\\r_j\equiv d+1-j\pmod 2}} \frac{2^{d+1}}{r_1^{s_1}\cdots r_{d}^{s_d}}\sum_{1\le r_{d+1}'\le (r_{d}-1)/2}\frac{1}{(2r_{d+1}')^{k+1}} \qquad(\text{since $r_{d}$ is odd}) \\
&\equiv\sum_{\substack{p>r_1>\cdots >r_{d}>0\\r_j\equiv d+1-j\pmod 2}} \frac{2^{d+1}}{r_1^{s_1}\cdots r_{d}^{s_d}} \cdot\frac{2^{-k-1}}{N}\sum_{j=0}^{N-1}(-1)^j\binom{N}{j}B_j\left(\frac{r_{d}-1}{2}\right)^{N-j}.
\end{align*}
Expanding the inner sum by adding and subtracting the extra term corresponding to $j=N$, we have
\begin{align*}
&\,\sum_{j=0}^{N}(-1)^j\binom{N}{j}B_j\left(\frac{r_{d}-1}{2}\right)^{N-j}-(-1)^N B_N\\
=&\,
\sum_{j=0}^{N}(-1)^j\binom{N}{j}B_j2^{j-N}\sum_{i=0}^{N-j}\binom{N-j}{i}(-1)^{N-j-i}r_{d}^{i}-(-1)^N B_N\\
=&\, \sum_{0\le i+j\le N} (-1)^{k+i}\binom{N}{j}\binom{N-j}{i}B_j2^{j-N}r_{d}^{i} -(-1)^N B_N\\
=&\, \sum_{0\le i+j\le N} (-1)^{k+i}\binom{N}{i}\binom{N-i}{j}B_j2^{j-N}r_{d}^{i} -(-1)^N B_N\\
=&\, \sum_{j=0}^{N}2^{-j}\Bigg\{\sum_{i=0}^{N-j}\binom{N-j}{i}B_i2^{i-(N-j)}\Bigg\}(-1)^{k+j}\binom{N}{j}r_{d}^{j}-(-1)^N B_N \quad(\text{swapping $i$ and $j$})\\
=&\, \sum_{j=0}^{N}2^{-j}B_{N-j}\Big(\frac12\Big)(-1)^{k+j}\binom{N}{j}r_{d}^{j}-(-1)^N B_N\\
=&\, \sum_{j=0}^{N}2^{-(N-j)}B_{j}\Big(\frac12\Big)(-1)^{k+N-j}\binom{N}{j}r_{d}^{N-j}-(-1)^N B_N \\
=&\, \sum_{j=0}^{N-1}2^{-(N-j)}B_{j}\Big(\frac12\Big)(-1)^{k+N-j}\binom{N}{j}r_{d}^{N-j}
+(-1)^k\Big(B_{N}\Big(\frac12\Big)- B_N\Big) .
\end{align*}
From this we have
\begin{align*}
S_\calA(\bfs,k+1)_p\equiv&\,
\frac{2^{-k}}{p-k-1}(-1)^k(2^{1+k}-2)B_{p-k-1}\sum_{\substack{p>r_1>\cdots >r_{d}>0\\r_j\equiv d+1-j\pmod 2}}\frac{2^{d}}{r_1^{s_1}\cdots r_d^{s_{d}}} \\ &\,+\frac{2^{-k}}{p-k-1}\sum_{j=0}^{p-k-2}2^{k+j}B_j\left(\frac12\right)\binom{p-k-1}{j}
\sum_{\substack{p>r_1>\cdots >r_{d}>0\\r_j\equiv d+1-j\pmod 2}}\frac{2^{d}}{r_1^{s_1}\cdots r_{d-1}^{s_{d-1}}}r_d^{p-k-1-j-s_{d}}\\
&\equiv (-1)^{k-1}T_\calA(k+1)_pT_\calA(\bfs)_p
+\frac{1}{p-k-1}\sum_{j=0}^{p-k-2}2^jB_j\left(\frac12\right)\binom{p-k-1}{j} T_{\calA}(\mathbf s+k+j)_p,
\end{align*}
by Corollary~\ref{cor:unitT-oddWt}. We can even replace $(-1)^{k-1}$ by $-1$ since this term is 0 when $k$ is odd.
\end{proof}

\begin{thm} \label{thm:v-DualReversal}
For all positive integers $n$ and $k$, we have
\begin{align}\label{equ:v-DualRev_n1k}
&T_\calA(n,\{1\}^{k})=T_\calA(k+1,\{1\}^{n-1}).
\end{align}
In particular,
\begin{align} \label{equ:v-DualRevEven-n1k}
T_\calA(2n,\{1\}^{k})=&\, (-1)^{k}T_\calA(\{1\}^{2n-1},k+1),\\
T_\calA(2n+1,\{1\}^k)=&\, (-1)^{k+1}S_\calA(\{1\}^{2n},k+1)  \label{equ:v-DualRevOdd-n1k}.
\end{align}
\end{thm}

\begin{proof}
It is clear that both \eqref{equ:v-DualRevEven-n1k} and \eqref{equ:v-DualRevOdd-n1k} follow from \eqref{thm:v-DualReversal} by the reversal relations. In what follows, we will prove \eqref{equ:v-DualRev_n1k} holds for each prime $p$-component ($p>n$ and $p>k+1$), i.e.,
\begin{align}\label{equ:v-DualRev_n1kp}
&T_\calA(n,\{1\}^{k})_p=T_\calA(k+1,\{1\}^{n-1})_p,
\end{align}
by induction on $n$, noting that the base case of $n=1$ is given by Remark~\ref{rem:unitT-oddWt}.

Now we assume $n\ge 2$ and $T_\calA(n-1,\{1\}^{k})_p=T_\calA(k+1,\{1\}^{n-2})_p$ for all $k\in\N$.

\medskip
\noindent
\textbf{Case (i). $n$ is even.} Then Lemma~\ref{lem:S(x)} and the reversal relation of $S_\calA(\{1\}^{n-2},k+j+1)$ yield that
\begin{align*}
T_\calA (\{1\}^{n-1},k+1)_p&\equiv \frac{-1}{p-k-1}\sum_{j=0}^{p-k-2} \binom{p-k-1}{j}2^jB_j\Big(\frac12\Big)S_\calA(\{1\}^{n-2},k+j+1)_p\\
&\equiv \frac{(-1)^{k}}{p-k-1}\sum_{j=0}^{p-k-2} \binom{p-k-1}{j}2^jB_j\Big(\frac12\Big)T_\calA(k+j+1,\{1\}^{n-2})_p\\
&\equiv \frac{(-1)^{k}}{p-k-1}\sum_{j=0}^{p-k-2} \binom{p-k-1}{j}2^jB_j\Big(\frac12\Big)T_\calA(n-1,\{1\}^{k+j})_p
\end{align*}
by induction assumption. By Lemma~\ref{lem:F(x)}, we see that
\begin{align*}
&T_\calA(\{1\}^{n-1},k+1)_p \equiv (-1)^{k} T_\calA(n,\{1\}^{k})_p
\end{align*}
which gives the even $n$ case of \eqref{equ:v-DualRev_n1kp}.

\medskip
\noindent
\textbf{Case (ii). $n$ is odd.}  Since $T_\calA(\{1\}^{n-1})_p=0$ by Theorem~\ref{thm:unitTA}, Lemma~\ref{lem:S2(x)} yields that
\begin{align*}
S_\calA(\{1\}^{n-1},k+1)_p
&\equiv\frac{1}{p-k-1}\sum_{j=0}^{p-k-2}2^{j}B_{j}\Big(\frac12\Big)(-1)^{j}\binom{p-k-1}{j}T_\calA(\{1\}^{n-2},k+j+1)_p \\
&\equiv\frac{1}{p-k-1}\sum_{j=0}^{p-k-2}2^{j}B_{j}\Big(\frac12\Big)(-1)^{j}\binom{p-k-1}{j}T_\calA(\{1\}^{k+j},n-1)_p \\
&\equiv\frac{(-1)^{k}}{p-k-1}\sum_{j=0}^{p-k-2}2^{j}B_{j}\Big(\frac12\Big)\binom{p-k-1}{j} T_\calA(n-1,\{1\}^{k+j})_p
\end{align*}
by the inductive hypothesis and then the reversal relation with $n-1$ is even. Applying \eqref{equ:keyReduction} gives rise to
\begin{align*}
S_\calA(\{1\}^{n-1},k+1)_p & \equiv (-1)^{k+1} T_\calA(n,\{1\}^{k})_p.
\end{align*}
Since $n$ is odd, applying the reversal relation on the left-hand side we get
\begin{align*}
T_\calA(k+1,\{1\}^{n-1})_p & \equiv T_\calA(n,\{1\}^{k})_p.
\end{align*}
This confirms \eqref{equ:v-DualRev_n1kp} when $n$ is odd.

\medskip
Combining Cases (i) and (ii) we can complete the proof of \eqref{equ:v-DualRev_n1kp} hence the theorem.
\end{proof}

The following corollary confirms the first relation in Conjecture~\ref{conj:T21k}. Using the value $T_\calA(k+2)=2(1-2^{-1-k}) \beta_{k+2}$ (see \cite[(37)]{Zhao2024a}) we can then find the finite analog of Corollary \ref{cor:SMTV21k}.

\begin{cor}\label{cor:TA21_k}
For all positive integers $k$, we have
\begin{equation}\label{equ:FMTV21kGeneral}
T_{\calA}(2,\{1\}^k)=(-1)^k T_{\calA}(1,k+1).
\end{equation}
In particular, if $k$ is odd then
\begin{equation}\label{equ:FMTV21k}
T_{\calA}(2,\{1\}^k)=-T_{\calA}(1,k+1)=(k+2)(2-2^{-1-k})\beta_{k+2} =\frac{(k+2)(2^{k+2}-1)}{2(2^{k+1}-1)}T_\calA(k+2).
\end{equation}
More generally (but in fact equivalently), for all $0\le\ell\le k$ with odd $k$, we have
\begin{equation}\label{equ:TA1ell21k}
T_\calA(\{1\}^\ell,2,\{1\}^{k-\ell})=(-1)^{\ell}\binom{k+2}{\ell+1} (2-2^{-1-k}) \beta_{k+2}.
\end{equation}
\end{cor}

\begin{proof}
Clearly, \eqref{equ:FMTV21kGeneral} is a special case of Theorem~\ref{thm:v-DualReversal}, which gives the first equality in \eqref{equ:FMTV21k}. When $k$ is odd, \eqref{equ:TA1ell21k} follows from Prop.~\ref{prop:FMT1ell2lk}, which implies the second equality
in \eqref{equ:FMTV21k}. The last equality of \eqref{equ:FMTV21k} follows from \cite[(42)]{Zhao2024a}.
\end{proof}

\begin{thm}\label{thm:ht2FMTVdual}
For all $n,k\in \mathbb N$ such that $n+k\equiv 1\pmod 2$, we have
\begin{equation*}
T_\calA(n,2,\{1\}^{k})=T_\calA(k+1,2,\{1\}^{n-1}).
\end{equation*}
Equivalently,
\begin{align}\label{equ:ht2FMTVdual-nOdd}
T_\calA(2n-1,2,\{1\}^{2k})=&\, -T_\calA(\{1\}^{2n-2},2,2k+1),\\
T_\calA(2n,2,\{1\}^{2k-1})=&\, -S_\calA(\{1\}^{2n-1},2,2k).\label{equ:ht2FMTVdual-nEven}
\end{align}
\end{thm}

\begin{proof}
We proceed according to the parity of $n$ so that we only need to prove \eqref{equ:ht2FMTVdual-nOdd} and  \eqref{equ:ht2FMTVdual-nEven} due to the reversal relations.

\medskip
\noindent
\textbf{Case (i). $n$ is odd.} Set $\bfs=(n,2)$. Then
$$T_\calA(\bfs-1,\{1\}^{k+j})=T_\calA(n,\{1\}^{k+j+1})=(-1)^{k+j}S_\calA(\{1\}^{n-1},k+j+2)$$
by Theorem~\ref{thm:v-DualReversal}. Noting that $k$ is even, we know that by Lemma~\ref{lem:F(x)}
\begin{align*}
T_\calA(n,2,\{1\}^{k})_p&\equiv\frac{-1}{p-k-1}\sum_{j=0}^{p-k-3}2^jB_j\Big(\frac12\Big) \binom{p-k-1}{j} T_\calA(n,\{1\}^{k+j+1})_p\\
&\equiv\frac{-1}{p-k-1}\sum_{j=0}^{p-k-3}2^jB_j\Big(\frac12\Big) \binom{p-k-1}{j}S_\calA(\{1\}^{n-1},k+j+2)_p   \pmod{p}.\\
&\equiv -T_\calA(\{1\}^{n-1},2,k+1)_p   \pmod{p}
\end{align*}
for all primes $p>k+1$ by Lemma~\ref{lem:S(x)}. This proves \eqref{equ:ht2FMTVdual-nOdd}.

\medskip
\noindent
\textbf{Case (ii). $n$ is even.} Since $k$ is odd, Lemma~\ref{lem:F(x)} implies that
\begin{align*}
T_\calA(n,2,\{1\}^{k})_p&\equiv\frac{1}{p-k-1}\sum_{j=0}^{p-k-3}2^jB_j\Big(\frac12\Big)\binom{p-k-1}{j}T_\calA(n,\{1\}^{k+j+1})_p\\
&\equiv\frac{-1}{p-k-1} \sum_{j=0}^{p-k-2}2^jB_{j}\Big(\frac12\Big)\binom{p-k-1}{j}T_\calA(\{1\}^{n-1},k+j+2)_p  \quad\text{(reversal)}\\
&\equiv - S_\calA(\{1\}^{n-1},2,k+1)_p
\end{align*}
for all primes $p>k+1$ by Lemma~\ref{lem:S2(x)} since $T_\calA(k+1)_p=0$. This confirms \eqref{equ:ht2FMTVdual-nEven}.
\end{proof}

We would like to conclude the paper by the following conjecture which is supported by overwhelming numerical evidence.

\begin{conj}\label{conj:allDuals}
We have
\begin{align*}
T_\calA(\bfs)=&\, T_\calA(\bfs^\vee),\quad  T_\sha^\Sy(\bfs)\equiv T_\sha^\Sy(\bfs^\vee) \pmod{\zeta(2)\ES_{w-2}\cap \SMTV_{w}}
\end{align*}
for the following $\bfs$:
\begin{enumerate}
    \item [\upshape{(1)}] $\bfs=(\{1\}^l,3,\{1\}^{k-l})$, where $k$ and $l$ are odd and $1\le l\le k$;
    \item [\upshape{(2)}] $\bfs=(\{1,3\}^k,1,1)$, where $k\in\N$.
\end{enumerate}
Moreover, the symmetric MTV counterpart of the results in Theorems~\ref{thm:v-DualReversal} and \ref{thm:ht2FMTVdual} should hold, i.e., for all positive integers $n$ and $k$,
\begin{align*}
T_\sha^\Sy(n,\{1\}^{k})\equiv&\, T_\sha^\Sy(k+1,\{1\}^{n-1}) \pmod{\zeta(2)\ES_{n+k-2}\cap \SMTV_{n+k}},\\
T_\sha^\Sy(n,2,\{1\}^{k})\equiv &\, T_\sha^\Sy(k+1,2,\{1\}^{n-1}) \pmod{\zeta(2)\ES_{n+k}\cap \SMTV_{n+k+2}} \quad\text{if $n+k$ is odd}.
\end{align*}

\end{conj}

\bigskip
\noindent
\textbf{Acknowledgment.} J. Zhao is supported by the Jacobs Prize from The Bishop's School.

\section*{Appendix. Space generated by symmetric multiple $T$-values}
\begin{center}
   {\large by Adam Fang, Kailin Xuan, Rudy Zhang, and Jianqiang Zhao}
\end{center}

In this appendix, we analyze Conjecture~\ref{conj:KanekoZagierMTV} in details. First,
we point out that unlike the Euler sum case in Conjecture~\ref{conj:KanekoZagierAltVersion}, one cannot replace $\SMTV_w$ by $\MTV_{w}$ because the definition of $T_\sha^\Sy(\bfs)$ involves both MTVs and MSVs when the depth of $\bfs$ is odd.

Numerical computation shows that not every symmetric MTV lies in $\MTV$. Already for $w=1$, $\SMTV_1=\SMSV_1= \log 2 \Q$ while $\MTV_1=\emptyset$. Moreover, the numerically computed dimensions for $\MTV$ in \cite{KanekoTs2020} and for $\FMTV$ in \cite{Zhao2024a} suggest that $\SMTV_w$ is much bigger than $\MTV_w$, which are both subspaces of the Euler sum space $\ES_w$. In fact, as a byproduct of our numerical verification of Conjecture~\ref{conj:KanekoZagierAltVersion} up to weight 10, we found the following conjecture.

\begin{conj}
For all weight $w\ge 1$, we have
\begin{align*}
    \MTV_w\subseteq \SMTV_w.
\end{align*}
\end{conj}

One should compare this to the result that symmetric Euler sums generate the whole space of Euler sums (see \cite{Anzawa2024}). We note that for weight $w\ge 2$, by \eqref{equ:SMTdp1}
\begin{align*}
    T_\sha^\Sy(w)=(1+(-1)^w -2^{1-w}+(-1)^w 2^{1-w} )\zeta(w).
\end{align*}
Hence,
\begin{align*}
    T(w)=(2-2^{1-w})\zeta(w)=\frac{2-2^{1-w}}{1+(-1)^w -2^{1-w}+(-1)^w 2^{1-w}}T_\sha^\Sy(w)\in \SMTV_w.
\end{align*}
But even in depth two and even weight case, it does not seem to be very easy to verify that
$T(a,b)\in \SMTV_{a+b}$ for all $a\ge 2, b\ge 1.$ However, if $a+b$ is odd then $T(a,b)$ is reducible to
$\Q$-linear combinations of depth one values so that it clearly becomes a rational multiple of $\zeta(a+b)$.

By combining the data from \cite[Table A.1]{Zhao2024a} and further computations, we have obtained Table~\ref{Table:dimMMV}.

\begin{table}[!h]
{
\begin{center}
\begin{tabular}{  |c|c|c|c|c|c|c|c|c|c|c|c|  } \hline
       $w$         & 0 & 1 &  2  &  3  &  4  &  5  &  6  & 7 & 8  &  9  &  10  \\ \hline
$\dim_\Q \FMTV_w$  & 0 & 1 &  0  &  1  &  2  &  3  &  3  & 6 & 9  &  15 &  17  \\ \hline
$\dim_\Q \SMTV_w$  & 0 & 1 &  1  &  1  &  3  &  4  &  6  & 9 & 16 &  22 &  32  \\ \hline
$\dim_\Q \MTV_w$   & 1 & 0 &  1  &  1  &  2  &  2  &  4  & 5 & 9  &  10 &  19  \\ \hline
$\dim_\Q \ES_w$    & 1 & 1 &  2  &  3  &  5  &  8  &  13 & 21& 34 &  55 &  89  \\ \hline
\end{tabular}
\end{center}
}
\caption{Conjectural dimensions (finite/symmtric) MTVs and MSVs.}
\label{Table:dimMMV}
\end{table}

We notice that $\dim\FMTV_3=\dim\SMTV_3-\dim\SMTV_1+1$ since $\zeta(2)S^\Sy(1)\not\in \SMTV_3$. Thus,
\begin{equation*}
 \FMTV_3\cong \frac{\SMTV_3}{\zeta(2)\SMTV_1\cap \SMTV_3}.
\end{equation*}
Similarly, $\dim\FMTV_7=\dim\SMTV_7-\dim\SMTV_5+1$ since $\zeta(2)S^\Sy(3,1,1)\not\in \SMTV_7$. Hence,
\begin{equation*}
 \FMTV_7\cong \frac{\SMTV_7}{\zeta(2)\SMTV_5\cap \SMTV_7}.
\end{equation*}
A new phenomenon appears in weight 8. We find that $\dim\FMTV_8=\dim\SMTV_8-\dim\SMTV_6-1$.
Since $\zeta(2)\SMTV_6\subset \SMTV_8$ we see that
\begin{equation*}
 \FMTV_8\not \cong \frac{\SMTV_8}{\zeta(2)\SMTV_6\cap \SMTV_8}.
\end{equation*}
We can guess that we need one more product with $\zeta(2)$. Indeed, we find that
\begin{multline*}
2137T_\calA(7, 1) + 246T_\calA(6, 2) + 570T_\calA(6, 1, 1) - 504T_\calA(5, 1, 1, 1) - 112T_\calA(3, 1, 3, 1) \\
- 28T_\calA(3, 1, 1, 2, 1) - 210T_\calA(1, 6, 1) + 28T_\calA(1, 3, 1, 2, 1) + 28T_\calA(1, 1, 3, 2, 1)=0
\end{multline*}
while
\begin{multline*}
2137T_\sha^\Sy(7, 1) + 246T_\sha^\Sy(6, 2) + 570T_\sha^\Sy(6, 1, 1) - 504T_\sha^\Sy(5, 1, 1, 1) - 112T_\sha^\Sy(3, 1, 3, 1) \\
- 28T_\sha^\Sy(3, 1, 1, 2, 1) - 210T_\sha^\Sy(1, 6, 1) + 28T_\sha^\Sy(1, 3, 1, 2, 1) + 28T_\sha^\Sy(1, 1, 3, 2, 1)\not\in \zeta(2)\SMTV_6.
\end{multline*}
However,
\begin{multline*}
2137T_\sha^\Sy(7, 1) + 246T_\sha^\Sy(6, 2) + 570T_\sha^\Sy(6, 1, 1) - 504T_\sha^\Sy(5, 1, 1, 1) - 112T_\sha^\Sy(3, 1, 3, 1) \\
- 28T_\sha^\Sy(3, 1, 1, 2, 1) - 210T_\sha^\Sy(1, 6, 1) + 28T_\sha^\Sy(1, 3, 1, 2, 1) + 28T_\sha^\Sy(1, 1, 3, 2, 1) \in \zeta(2)\SMTV_6\oplus \langle \zeta(2)\zeta(2,\bar1,\bar3)\rangle.
\end{multline*}
Hence, we still have
\begin{equation*}
 \FMTV_8 \cong \frac{\SMTV_8}{\zeta(2)\ES_6\cap \SMTV_8}.
\end{equation*}

As an example in weight 9, we have
\begin{align*}
&\, 367200T_\calA(\{1\}^5, 3, 1)
=12270497T_\calA(9) - 3047760T_\calA(7, 1, 1) - 543660T_\calA(6, 2, 1) - 2334440T_\calA(6, 1, 2)\\
&\,+ 252960T_\calA(6, \{1\}^3) - 2040T_\calA(5, 3, 1) + 71060T_\calA(5, 2, 2) + 550800T_\calA(1, 7, 1) + 30600T_\calA(1, 6, 2) \\
&\,+ 97920T_\calA(1, 5,  \{1\}^3) + 48960T_\calA(1, 4, 2, 1, 1)
\end{align*}
while
\begin{align*}
&\,367200T_\sha^\Sy(\{1\}^5, 3, 1)
=12270497T_\sha^\Sy(9) - 3047760T_\sha^\Sy(7, 1, 1) - 543660T_\sha^\Sy(6, 2, 1) - 2334440T_\sha^\Sy(6, 1, 2)\\
&\,+ 252960T_\sha^\Sy(6, \{1\}^3) - 2040T_\sha^\Sy(5, 3, 1) + 71060T_\sha^\Sy(5, 2, 2) + 550800T_\sha^\Sy(1, 7, 1) + 30600T_\sha^\Sy(1, 6, 2) \\
&\,+ 97920T_\sha^\Sy(1, 5,  \{1\}^3) + 48960T_\sha^\Sy(1, 4, 2, 1, 1) -  \frac{25080150881}{1764 }\zeta(2)T_\sha^\Sy(7)+ \frac{325260422}{49}\zeta(2)T_\sha^\Sy(6, 1))\\
&\,+ \frac{87870127}{98}\zeta(2)T_\sha^\Sy(5, 2))+ 2154240\zeta(2)T_\sha^\Sy(5, 1, 1) + 605880\zeta(2)T_\sha^\Sy(4, 2, 1) - 100640\zeta(2)T_\sha^\Sy(4,  \{1\}^3) \\
&\,- 269280\zeta(2)T_\sha^\Sy(1, 5, 1) - 67320\zeta(2)T_\sha^\Sy(1, 4, 2) - 269280\zeta(2)T_\sha^\Sy(1, 3, \{1\}^3).
\end{align*}
We found that the nine $\zeta(2)$-products in the above generate $\zeta(2)\SMTV_7$ whose dimension should be exactly equal to 9. Among them,
\begin{equation*}
\zeta(2)T_\sha^\Sy(4,  \{1\}^3), \zeta(2)T_\sha^\Sy(1, 5, 1)\not\in \SMTV_9,
\end{equation*}
but
\begin{equation*}
\zeta(2)\SMTV_7\subseteq \SMTV_9\oplus \big\langle\zeta(2)T_\sha^\Sy(4,  \{1\}^3), \zeta(2)T_\sha^\Sy(1, 5, 1)  \big\rangle.
\end{equation*}
Moreover, we have verified numerically that $\dim \zeta(2)\ES_7\cap \SMTV_9=7$ and
\begin{equation*}
 \FMTV_9 \cong \frac{\SMTV_9}{\zeta(2)\ES_7\cap \SMTV_9}.
\end{equation*}
Similarly, in weight 10, we see that $V=\big\langle\zeta(2)T_\sha^\Sy(1, 1, 3, 2, 1)\big\rangle\not\subset \SMTV_{10}$ and  $\zeta(2)\SMTV_8\subset \SMTV_{10}\oplus V.$ Hence, we can verify numerically that
\begin{equation*}
 \FMTV_{10} \cong \frac{\SMTV_{10}}{\zeta(2)\ES_8\cap \SMTV_{10}}.
\end{equation*}
All the evidence above strongly supports Conjecture~\ref{conj:KanekoZagierMTV}.

We now list all symmetric MTVs with weight up to 5. Note that we can omit some of them using the relation $T_\sha^\Sy(\bfs)=(-1)^{|\bfs|}T_\sha^\Sy(\overleftarrow{\bfs})$ when the depth of $\bfs$ is even. Here, $\overleftarrow{\bfs}$ is the reversal of $\bfs.$

\begin{alignat*}{4}
T_\sha^\Sy(1)=&\, 2\log 2,\quad &  T_\sha^\Sy(4)=&\,\frac{16}{15}T(4),\\
T_\sha^\Sy(2)=&\, \frac43 T(2),\quad &   T_\sha^\Sy(3,1)=&\,  -\frac34T(4)-\frac12T(2,2),\\
T_\sha^\Sy(1,1)=&\, 0,\quad &  T_\sha^\Sy(2,2)=&\, 3T(4)+2T(2,2),\\
T_\sha^\Sy(3)=&\, \frac67 T(3),\quad &  T_\sha^\Sy(2,1,1)=&\,  \frac{13}{4}T(4)-\frac12 T(2,2), \\
T_\sha^\Sy(2,1)=&\,  3T(3),\quad &  T_\sha^\Sy(1,1,2)=&\,  \frac{1}{2}T(4)+T(2,2)+3S(2,1,1), \\
T_\sha^\Sy(1,1,1)=&\, \frac67 T(3),\qquad \qquad &  T_\sha^\Sy(1,2,1)=&\,  -\frac{9}{4}T(4)+\frac12T(2,2)-3S(2,1,1).
\end{alignat*}
Note that $S(2,1,1)\not\in \MTV_4=\langle T(4),T(2,2)\rangle$ by numerical computation.
Let $a=(S(2,2,1)+M(2,\check{2},\check{1}))/2$ and $b=3S(4,1)/2-S(3,1,1)$. Then the weight five values are:
\begin{alignat*}{4}
T_\sha^\Sy(5)=&\,\frac{30}{31}T(5),\quad &  T_\sha^\Sy(2,2,1)=&\, 2T(5)-4T(4,1)+2a, \\
T_\sha^\Sy(4,1)=&\, T(5)+4T(4,1),\quad &  T_\sha^\Sy(1,2,2)=&\, -5T(5)+22T(4,1)-a-4b, \\
T_\sha^\Sy(3,2)=&\, -2T(5)-8 T(4,1),\qquad \qquad & T_\sha^\Sy(3,1,1)=&\, -\frac{1}{4}T(5)-a, \\
T_\sha^\Sy(2,1,2)=&\, T(5)+4T(4,1), \quad &
T_\sha^\Sy(1,1,3)=&\,  \frac{5}{2}T(5)-\frac{15}2 T(4,1)+b.
\end{alignat*}
The remaining weight 5 symmetric MTVs all have height 0 or 1 and thus are all rational multiples of $T(5)$ given by Theorem~\ref{thm:unitSMT} and Theorem~\ref{thm:SMT121}.

\end{document}